\documentclass[12pt]{amsart}
\usepackage{amsmath, amsthm, amssymb}
\usepackage{enumerate}
\usepackage{color}

\usepackage{tikz}
\usetikzlibrary{arrows,decorations.markings}

\tikzset{vertex/.style={circle,draw,fill,inner sep=0pt,minimum size=1mm}}
\tikzset{oriented edge/.style={thick, postaction={decorate,
         decoration={markings,mark=at position .5 with \arrow{angle 90};}}}}

\newtheorem{lem}{Lemma}
\newtheorem{thm}{Theorem}
\newtheorem*{ex}{Example}

\newcommand{\OMM}{\mathrm{OMM}}
\newcommand{\tr}{\mathrm{tr}}

\begin{document}

\title [Periods of orbits for maps on graphs] {Periods of orbits for maps on graphs homotopic to a constant map}

\author[Bernhardt]{Chris Bernhardt}

 \address{Department of mathematics and computer science\\Fairfield
University\\Fairfield\\CT 06824}
\email{cbernhardt@fairfield.edu}

 \author[Gaslowitz]{Zach Gaslowitz}
 \address{Department of mathematics\\Harvey Mudd College\\340 East Foothill Blvd.\\
Claremont\\California 91711}
 \email{Zachary\_Gaslowitz@hmc.edu}
 
 \author[Johnson]{Adriana Johnson}
 
 \address{Department of mathematics\\ Bard College\\PO Box 5000\\Annandale-on-Hudson\\NY 12504-5000}
 \email{aj957@bard.edu}
 
 \author[Radil]{Whitney Radil}
 
 \address{ Department of mathematics\\College of Saint Benedict\\
37 South College Avenue\\St. Joseph, Minnesota 56374}
\email{wiradil@csbsju.edu}

 \thanks{*This work was done as part of an REU at Fairfield University supported by the National Science Foundation and the Department of Defense under Grant No. 1004346.}

 \thanks{The authors thank the referees for valuable suggestions and corrections that greatly improved the clarity of the paper.}


\begin{abstract}
The paper proves two theorems concerning the set of periods of periodic orbits for maps of graphs that are homotopic to the constant map and such that the vertices form a periodic orbit. The first result is that if the number of vertices  is not a divisor of $2^k$ then there must be a periodic point with period $2^k$.  
The second is that if the number of vertices is $2^ks$  for odd $s>1$, then for all $r>s$ there exists a periodic point of minimum period $2^k r$. These results are then compared to the Sharkovsky ordering of the positive integers.

\end{abstract}

\maketitle


\section{Introduction}

A vertex map on a graph with $v$ vertices is a continuous map that permutes the
vertices. Given a vertex map, the periods of the periodic orbits can
be computed giving a subset of the positive integers. One of the
basic questions of combinatorial dynamics for vertex maps is to
determine which subsets of the positive integers can be obtained in
this way. Sharkovsky's theorem \cite{S} is a well-known result about the periods of periodic orbits for maps on the real line or the interval. It provides the 
answer when the underlying graph is topologically an
interval and the vertices all belong to the same periodic orbit. In this case  the map must have a periodic orbit of period $m$ for any $m$ satisfying $m \triangleleft v$, where \[1 \triangleleft 2 \triangleleft 4 \triangleleft \dots \, \,  \dots \triangleleft 2^27 \triangleleft 2^25 \triangleleft 2^23 \triangleleft \dots 2^17 \triangleleft 2^15 \triangleleft 2^13 \dots 7 \triangleleft 5 \triangleleft 3.\]

In \cite{BGMY}, Block, Guckenheimer, Misiurewicz and Young gave what has now become the standard approach to proving Sharkovsky's theorem using directed graphs. Among other results that were proved in the paper was an extension of Sharkovsky's Theorem to degree zero maps of the circle. They showed that if a degree zero map of the circle has a periodic point of period $v$ then it must also have one of period $m$ for all $m \triangleleft v$. Again this can be considered as a result for vertex maps on graphs where the graph is topologically a circle and the periodic points of period $v$ form the vertices. (A good introduction to combinatorial one-dimensional dynamics that contains these results amongst a wealth of others is \cite{ALM}.)

After studying maps on the interval and on circles it was natural to ask similar questions for maps on trees and then on general graphs. In what follows we will always make the assumption that the underlying map permutes the vertices. A more general approach is not to have this restriction. This more general approach has been taken by a number of authors. In a sequence of papers culminating in \cite{AJM}, Alsed\`a,  Juher and Mumbr\'u proved the general result for trees. For maps on graphs using this more general approach see \cite{{AR}, {LL}, {LPR}}.

In \cite{B2,B3} a Sharkovsky-type theorem was proved  for vertex maps on trees, maps for which the vertices form one periodic orbit. This ordering is a partial ordering, not a linear ordering -- some of the relations in the Sharkovsky ordering have to be deleted. A natural question is to ask whether this ordering also holds for vertex maps of general graphs if we restrict the underlying map to be homotopic to a constant map. 

We have not been able to completely prove this, but we do prove two results for the case when the vertices form one periodic orbit with period $v$: the first result is that if $v$ is not a divisor of $2^k$ then there must be a periodic point with period $2^k$;  and the second is that if $v=2^ks$  for odd $s>1$, then for all $r>s$ there exists a periodic point of minimum period $2^k r$. In the final section we show that our results are quite strong. The set of periods forced by a given $v$ with respect to the results in our paper, those given by the tree ordering, and those given by the Sharkovsky ordering differ at most by a finite number of periods.

 In \cite{BGMY}, Block et al. proved their result for circles by looking at the universal cover and periodic orbits of lifts of the original map to the universal cover. This approach does not extend easily in the case considered in this paper. Though the universal cover is a tree, the vertices of the tree can be pre-periodic points and not periodic points. This means that we cannot simply apply the tree result from \cite{B3}. Our approach is similar to that in \cite{B1, B2, B3} using trace arguments for Oriented Markov Matrices.

There is a close connection between combinatorial dynamics and algebraic topology. As noted in \cite{B1} several of the ideas that we express in dynamical terminology could be expressed in terms of algebraic topology. In particular, the Oriented Markov Matrix introduced in section $5$ is the matrix corresponding to the induced map on $1$-chains, the vectors in section $6$ are the coordinate vectors for $1$-chains, and Theorem $1$ in section $7$ is closely related to the Lefschetz number of the map. However, we give elementary proofs for all results. We begin by introducing the basic concepts, starting with graphs.

\section{Graphs}

We are considering finite connected \emph{graphs}, whose \emph{edges} are real closed intervals, the endpoints of which are the \emph{vertices}.
Any two edges are pairwise disjoint, except possibly at their endpoints. We allow the possibility of more than one edge connecting the same two vertices, but we do not allow {\em loops}, edges that connect a vertex to itself. (What we are calling graphs are sometimes referred to as {\em multigraphs}.) We say a graph $G$ has $n$ edges and $v$ vertices.

For each edge $E_i$, an \emph{orientation} is assigned. Orientation is defined by the vertices that bound the edge: one will be considered the initial vertex, and the other the final vertex. If orientation is reversed, the edge will be written as $-E_i$, and the initial and final vertices are switched.

\begin{ex}
Let $\hat{G}$ be the graph represented below. Orientation is indicated by the arrows.

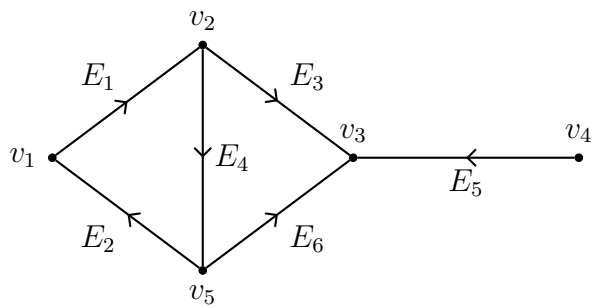
\begin{center}
	\begin{tikzpicture}[auto]
		\node[vertex] (one)   at (0,0)    [label=left :$v_1$] {};
		\node[vertex] (two)   at (2,1.5)  [label=above:$v_2$] {};
		\node[vertex] (three) at (4,0)    [label=above:$v_3$] {};
		\node[vertex] (four)  at (7,0)    [label=above:$v_4$] {};
		\node[vertex] (five)  at (2,-1.5) [label=below:$v_5$] {};
		
		\draw[oriented edge] (one)   to node      {$E_1$} (two);
		\draw[oriented edge] (five)  to node      {$E_2$} (one);
		\draw[oriented edge] (two)   to node      {$E_3$} (three);
		\draw[oriented edge] (two)   to node      {$E_4$} (five);
		\draw[oriented edge] (four)  to node      {$E_5$} (three);
		\draw[oriented edge] (five)  to node[swap]{$E_6$} (three);
	\end{tikzpicture}
\end{center}

\noindent
We will be using $\hat{G}$ in the examples that follow.
\end{ex}

Given any two vertices, $v_a$ and $v_b$, a \emph{path} from $v_a$ to $v_b$ is a sequence of edges $E_{i_1}, \ldots, E_{i_q}$ where the initial vertex of $E_{i_1}$ is $v_a$, the final vertex of $E_{i_q}$ is $v_b$ and the final vertex of $E_{i_r}$ is the initial vertex of $E_{i_{r+1}}$ for $1 \le r < q$.
If $E_p$ and $-E_p$ are two consecutive edges in a path, we can obtain a shorter path by omitting these two edges.
We will call this a \emph{contraction} of the path.
Given any path from vertex $v_a$ to vertex $v_b$ we can form a sequence of contractions resulting in a unique path that cannot be contracted further.
We call this resulting path \emph{reduced}. 

We adopt the standard graph theory term for a \emph{cycle}. A \emph{cycle} in a graph $G$ is a closed path with no repeated vertices or edges.

\section{Maps}
In this paper, we will consider continuous maps $F$  that act on $G$ and permute the vertices by a permutation $\theta$. 
We say the permutation $\theta$ is a \emph{cycle} if the vertices of $G$ form one periodic orbit. We will use cycle notation for permutations. Thus $(1,2,3,4)$ means $1$ gets mapped to $2$, $2$ to $3$, $3$ to $4$ and $4$ back to $1$.

Each edge in the graph is homeomorphic to the unit interval. We use the homeomorphism to define the distance between points in an edge and to give each edge unit length. A path consisting of $m$ edges is defined to have length $m$ in the obvious way. Suppose that an edge $E_i$ is mapped by $F$ to a path with $m$ edges, then there is a natural induced map $F^*:[0,1] \to [0,m]$. We will say that $F$ is linear on $E_i$ if $F^*$ is linear.

We now define the {\em linearization} of the map $F$, which we will denote by $f$.  For all vertices $v \in G$, we define $f(v)=F(v)$. If $E_i$ is an edge with endpoints $v_a$ and $v_b$, we define $f$ to map $E_i$ linearly onto the reduced path from $v_a$ to $v_b$ that is obtained from $F(E_i)$.

More formally, let $[0,1]=I$, we define $f:G  \to G$ to be the {\em linearization} of $F$ if for each edge $E_i$ there is homotopy $h_i:E_i \times I \to G $ which has the following properties : $h_i(x,0)=F(x)$ for all $x \in E_i$; $h_i(x,1)=f(x)$ for all $x \in E_i$; $h_i(v_a, t)=F(v_a)=f(v_a)$ for all $t \in I$; $h_i(v_b, t)=F(v_b)=f(v_b)$ for all $t \in I$; and such that $f$ is linear on $E_i$.
If $f$ is the linearization of a map $F$, we say $f$ is \emph{linearized}.
We will consider linearized maps throughout the paper. (In the literature the maps that we are calling linearized are sometimes referred to as {\em linear models} for tree maps or {\em connect-the-dots} maps for interval maps, see \cite{AGLMM, ALM}.) The connections between the sets of periodic points for $f$ and $F$ will be discussed in the next section.

In addition to thinking of $f$ as a map from $G$ to itself, we can also consider it as a map from paths in $G$ to paths in $G$. In this sense, the image of an edge will be a path. The image of a reversed edge is the reverse path.


\begin{ex}
We continue with our example. Suppose $f: \hat{G} \to \hat{G}$ permutes the vertices by $\theta=(v_1,v_2,v_3,v_4,v_5)$. Then $f(v_1)=v_2$, $f(v_2)=v_3$, etc. Though the images of the vertices are defined by $\theta$, there are many possibilities for the image of edges. For example, $f(E_1)$, considered as a path, must have endpoints $v_2$ and $v_3$, but its image could be any of the following:

\begin{align*}
f(E_1)&=E_3\\
&\text{OR}\\
&=E_4E_6\\
&\text{OR}\\
&=-E_1{-E_2}E_6{-E_3}E_4E_6\\
&etc.
\end{align*}
However, since $f$ is linearized, we could not have $f(E_1) = E_4E_6{-E_5}E_5$, since its image could be contracted further.

In what follows, we will take
\begin{align*}
\ \ \ \ \ \ \ \ \ \ \ \ \ \ \ \ \ \ \ \ \ \ \ \ &f(E_1)=E_3\\
&f(E_2)=-E_2E_6{-E_3}\\
&f(E_3)=-E_5\\
&f(E_4)=-E_6E_2\\
&f(E_5)=E_2E_1E_3{-E_6}{-E_4}{-E_1}{-E_2}E_6{-E_5}\\
&f(E_6)={-E_2}E_6{-E_5} \ .
\end{align*}

\end{ex}


\section{Periodic Points}
\label{sec:PeriodicPoints}
The map $f^r(x)$ refers to the map given by composing $f$ with itself $r$ times, so, for example, $f^2(x)=f(f(x))$.
We say that a point $x \in G$ is a \emph{periodic point} under $f$ if there exists a positive integer $p$ such that $f^p(x)=x$.
Any such $p$ is said to be a period of $x$, and the smallest period is known as that point's \emph{minimum period}.

Our main tool for finding periodic points is an application of the following lemma which stems from the the Intermediate Value Theorem.
\begin{lem}
For any closed real interval $I$, any continuous map from $I$ to itself will have a fixed point.
\end{lem}

Given a graph $G$ and a linearized map $f$ that permutes the vertices, we construct an {\em Oriented Markov Graph}. The vertices of the OMG correspond to the edges of $G$. A directed, oriented edge will be drawn from one OMG vertex, $E_j$ to another, $E_i$, if $E_j$ has a closed subinterval that maps entirely onto $E_i$. A directed edge will be drawn for each such closed subinterval. The orientation of the edge will be positive (resp. negative) if the subinterval gets mapped onto $E_j$ with positive (resp. negative) orientation. Though we will not use the term, in the literature, if $E_j$ contains a closed subinterval with image equal to $E_i$ it is said that $E_j$ {\em f-covers} $E_i$. Below we sketch standard results for Markov Graphs and refer the reader to \cite{ALM} for formal proofs that are stated in terms of $f$-covers.

Given an Oriented Markov Graph, we can define a sequence $E_{i_0}E_{i_1}\cdots E_{i_d}$ to be a \emph{walk} of length $d$ in the OMG, where  each edge $E_{i_k}$ in the sequence will have an edge connecting it to $E_{i_{k+1}}$ in OMG, or equivalently, there is a closed subinterval  of $E_{i_k}$ that gets mapped exactly onto $E_{i_{k+1}}$. We call a walk \emph{closed} if its first and last edge are equal, discounting orientation.

Closed walks are useful because if there is a closed walk of length $d$ from edge $E_k$ to itself in the OMG, then there is a periodic point of period $d$ in the graph. 
This is because if $E_{i_0}E_{i_1}\cdots E_{i_{d-1}}E_{i_d}$ is a closed walk with $E_{i_0}=E_{i_d}$
then we know that there is a subinterval $J_{d-1}$ in $E_{i_{d-1}}$ such that $f(J_{d-1})=E_{i_{d}}=E_{i_0}$. We can then find a subinterval $J_{d-2}$ in $E_{i_{d-2}}$ such that $f(J_{d-2})=J_{d-1}$. We proceed inductively until we obtain a subinterval $J_{0}$ of $E_{i_0}$ with the property that $f^d(J_0)=E_{i_{d}}=E_{i_0}$. Since $J_0 \subseteq E_{i_0}$ and $f^d(J_0)=E_{i_0}$, it follows from Lemma 1 that $f^d$ must have a fixed point in $E_{i_0}$. There could be more than one fixed point, of course, but there must be at least one. 
 
We have shown that closed walks in the Oriented Markov Graph give us information about the periodic points of the linearized map $f$. Suppose that $f$ is the linearization of $F$, then  these closed walks also give us information about periodic points of $F$. This follows from the fact that if there is a subinterval of $E_{i_k}$ that gets mapped exactly onto $E_{i_{k+1}}$ by $f$, then there must be a subinterval of $E_{i_k}$ that gets mapped exactly onto $E_{i_{k+1}}$ by $F$.

 All of our arguments for periodic orbits will consider these closed walks and will give both periodic orbits for $F$ and for $f$. It should be noted that the periodic points of $f$ correspond to closed walks in the OMG, but $F$ could have periodic points of other periods in addition to those given by closed walks. It should also be noted that in the proof above it is important that the edges are closed intervals and that they are not allowed to contain vertices in their interiors.


\section{Oriented Markov Matrices}
\label{sec:OMM}
We will denote the Oriented Markov Matrix of $f$ as $\OMM(f)=M$.
The Oriented Markov Matrix is an $n\times n$ matrix such that an element $a_{ij}$ represents the number of times the edge $E_j$ maps to the edge $E_i$ with positive orientation minus the number of times $E_j$ maps to $E_i$ with negative orientation. In terms of the OMG, $a_{ij}$ is equal to the number of positive directed edges from $E_j$  to $E_i$ minus the number of negative directed edges from $E_j$  to $E_i$.

\begin{lem}
If $M$ is the Oriented Markov Matrix of a map $f$, then the $ij^\text{th}$ entry of $M^k$ counts the number of positively oriented walks minus the number of negatively oriented walks from $E_j$ to $E_i$ of length $k$.
\end{lem}
\begin{proof}
It is clear from the construction that this holds when $k=1$.

Assume that it holds for $k=s$.
By definition,
\[
	(M^{s+1})_{ij}=(MM^s)_{ij}=\sum_{r=1}^n M_{ir}(M^s)_{rj}\ .
\]
By hypothesis, $(M^s)_{rj}$ counts the number, $p_s$, of positive walks of length $s$ from $j$ to $r$ minus the number of negative ones, $n_s$.
Similarly, $M_{ir}$ counts the number of positive walks, $p_1$, minus the number of negative walks, $n_1$, of length one from $r$ to $i$.
Thus,
\[
	M_{ir}(M^s)_{rj}=(p_1-n_1)(p_s-n_s)=(p_sp_1+n_sn_1)-(n_sp_1+p_sn_1)
\]
 counts the number of positively oriented length $s+1$ walks from $E_j$ to $E_i$ whose second to last step is $E_r$ minus the negatively oriented ones.

By summing over all possible $r$, we have counted all length $s+1$ walks from $E_j$ to $E_i$.
Thus, this lemma holds for $k=s+1$, so by induction, the lemma holds for all positive integers $k$.
\end{proof}

We will also need to use the following result. See \cite{B1} for proof.

\begin{lem}
If $M$ is the Oriented Markov Matrix for a map $f$, then $M^k$ is the Oriented Markov Matrix of $f^k$ for all positive integers $k$.
\end{lem}

These lemmas suggest that these Oriented Markov Matrices will be very useful tools in proving the existence of the walks that were discussed in section \ref{sec:PeriodicPoints} on periodic points.

\begin{ex}
\noindent
The corresponding Oriented Markov Matrix for $f:\hat G \to \hat G$ is
\begin{center}
$M = \begin{bmatrix}
0 & 0& 0& 0& 0& 0\\
0 &-1& 0& 1& 0&-1\\
1& -1& 0& 0& 1& 0\\
0 & 0& 0& 0&-1& 0\\
0 & 0&-1& 0&-1&-1\\
0 & 1& 0&-1& 0 & 1
\end{bmatrix}$\ .
\end{center}
\end{ex}

 If an element $a_{ij}$ in $M^r$ is non-zero, then there is at least one closed walk of length $r$ from edge $E_j$ to edge $E_i$. Non-zero entries in the diagonal of $M^r$ represent closed walks of length $r$ from an edge to itself. So the trace of $M^r$ represents the number of times edges in $G$ map to themselves with positive orientation minus the number of times they map to themselves with negative orientation with length $r$. We also note that whether or not an edge maps to itself in an orientation preserving or reversing way is independent of the orientation chosen for that edge. This means that the diagonal entries in powers of $M$ do not depend on the choice of the orientation for edges in $G$.


\section{Cycles}

In what follows we shall be studying maps from graphs to themselves that are homotopic to a constant map.
We will call such a map an \emph{HTC} map.

It is clear that a map from a graph $G$ is HTC if and only if the image of every cycle in the graph  can be contracted to the empty path.

\begin{ex}
\noindent
Note that there are three cycles in the graph: 

\begin{center}
$c_1=E_1E_4E_2$, $c_2=-E_3E_4E_6$, and $c_3=E_1E_3{-E_6}E_2$.
\end{center}

\noindent
 For $f$ to be homotopic to the constant map, we must show that the images of these cycles collapse.

\begin{align*}
f(c_1)&=f(E_1E_4E_2)\\
&=f(E_1)f(E_4)f(E_2)\\
&=(E_3)({-E_6}E_2)({-E_2}E_6{-E_3})\\
&=E_3{-E_6}E_2{-E_2}E_6{-E_3}
\end{align*}

\begin{center}
We can now  collapse the edges in this sequence:
\end{center}
\begin{align*}
\ \ \ \ &\sim E_3{-E_6}\mathbf{(E_2{-E_2})}E_6{-E_3}\\
\ \ \ \ &\sim E_3\mathbf{(-E_6E_6)}{-E_3}\\
\ \ \ \ &\sim \mathbf{(E_3{-E_3})}\\
\ \ \ \ &\sim \emptyset
\end{align*}
And the same can be done with the other cycles. So the map is HTC.

\end{ex}

To each path in the graph $G$ we associate an $n$-dimensional vector. The $k$-th component of the vector 
 counts the number of times the edge $E_k$ appears in the path with positive orientation minus the number of times $E_k$ appears in the path with negative orientation. Notice that if $\vec u$ is such a vector, then $M\vec u$ will give a vector that corresponds to the image of the path corresponding to $\vec u$ under $f$ .

For a tree, the number of edges is equal to one less than the number of vertices, $n=v-1$.
If $n>v-1$, then there is at least one cycle in the graph.

Given a graph with $n>v-1$, choose a spanning tree. For each edge that is not in the spanning tree we can form a cycle consisting of that edge and the remaining edges taken from the tree. 
It is clear that the vectors associated to these cycles  are linearly independent. In fact the vector associated to any closed path can be written as a linear combination of these $c=n-(v-1)$ vectors. This is a standard result from homology where it is seen that these vectors generate the group of $1$-cycles, see \cite{H}, for example.


We let $W=\{\vec{w_1}, \vec{w_2}, \cdots, \vec{w_c}\}$ denote the linearly independent set of vectors that correspond to these cycles. Since cycles collapse, their image is the empty path, and so $M\vec{w_j} = \vec{0}$.

\begin{ex}
In our example, the vectors corresponding to $c_1$ and $c_2$ give two linearly independent vectors, $\vec w_1^T=[1,1,0,1,0,0]$, $\vec w_2^T=[0,0,-1,1,0,1]$. The vector associated to $c_3$ is equal to $\vec w_1 - \vec w_2$. We can take $W=\{\vec w_1,\vec w_2\}$ as the linearly independent set of vectors.
\[
M \vec w_1 =
\begin{bmatrix}
0 & 0& 0& 0& 0& 0\\
0 &-1& 0& 1& 0&-1\\
1& -1& 0& 0& 1& 0\\
0 & 0& 0& 0&-1& 0\\
0 & 0&-1& 0&-1&-1\\
0 & 1& 0&-1& 0& 1
\end{bmatrix}
\begin{bmatrix}
1\\1\\0\\1\\0\\0
\end{bmatrix}
=\begin{bmatrix}
0\\0\\0\\0\\0\\0
\end{bmatrix}
\]

\end{ex}


\section{Trace theorems for Oriented Markov Matrices}

In this section we will prove results concerning the traces of the Oriented Markov Matrices of HTC maps on $G$. These results will be used in the following section to prove the main results. First, we state a result that we will need about maps on trees. This was proved in \cite{B1}. We give a proof here to aid the exposition.

\begin{lem}
Given a tree $T$ with $v$ vertices and a map $f:T \to T$ that permutes the vertices, if none of the vertices are fixed under $f$, then the trace of the Oriented Markov Matrix  is $-1$.
\end{lem}

\begin{proof}
For each vertex, $v_i$ there is a reduced path from
$v_i$ to  $f(v_i)$. Put a dot on the first edge in this path. 

Observe that an edge $E_i$ contains two dots if and only
if $-E_i$ is in the reduced path that corresponds to $f(E_i)$. Also observe that $E_i$ contains no
dots if and only if $E_i$ is in the reduced path corresponding to $f(E_i)$. Finally, an
edge contains one dot if and only if the reduced path of $f(E_i)$ does not
contain either $E_i$ or $-E_i$. Notice that the number of dots on
the edge $E_i$ is exactly $1-M_{ii}$. If $e$ denotes the number of edges in $T$,  the total number of
dots is $\sum_1^e(1-M_{ii})=e-\tr(M)$. However, there are
exactly $v$ dots on $T$, so
$v=e-\tr(M)$, and $\tr(M)=e-v=-1$.
\end{proof}

\begin{lem}
	Given any graph $G$ and 
	any permutation $\theta$ that does not fix any vertices, there exists an HTC map from $G$ to $G$ which permutes the vertices of $G$ by $\theta$ and has an Oriented Markov Matrix with trace $-1$.
\end{lem}

\begin{proof}
	Let $S$ be a spanning tree of $G$ and $f:G \to G$ any map that permutes the vertices according to $\theta$ and whose image is $S$.
	We know from the previous lemma that a map from a tree to itself that does not fix any vertex will have an Oriented Markov Matrix with trace $-1$, so $\tr(\OMM(f|_S))=-1$. 	The remaining edges are not in the image, so, they do not map to themselves.
	Thus, no other edges will contribute to the trace of $\OMM(f)$, so $\tr(\OMM(f))=-1$, as desired.
\end{proof}


\begin{lem}
	Given any graph $G$ and any permutation $\theta$, the Oriented Markov Matrix of any two HTC maps from $G$ to $G$ that permute the vertices by $\theta$ will have the same trace.
\end{lem}

\begin{proof}
Suppose the graph $G$ has $n$ edges, and $c$ linearly independent cycles. Let $W=\{\vec w_1, \vec w_2, \cdots, \vec w_c\}$ denote the set of linearly independent cycles on $G$.

%

Let $f$ and $g$ be two maps that are HTC and that permute the vertices by the same permutation $\theta$. For each edge $E_i$, the reduced paths corresponding to $f(E_i)$ and $g(E_i)$ have the same initial and terminal points. This means that if we let $\vec v_1$ and $\vec v_2$ demote the corresponding vectors, then  $\vec v_1 - \vec v_2$ will be a closed path (a homological $1$-cycle) and thus can be written as a linear combination of vectors from $W$.

 Denote the Oriented Markov Matrices of $f$ and $g$ by $M$ and $N$. Then   $M$ and $N$ can be related by $M=N+B$, where $B$ is an $n\times n$ matrix such that each column is an integer linear combination of the vectors in $W$. Therefore, $B=\bigg{[}\sum_{i=1}^{c} a_{1i}\vec w_i |\sum_{i=1}^{c} a_{2i}\vec w_i | \cdots |\sum_{i=1}^{c} a_{ni}\vec w_i\bigg{]}$. 
So, 
\begin{align*}
 \tr(B) &= \sum_{i=1}^{c} a_{1i}w_{1i} + \sum_{i=1}^{c} a_{2i}w_{2i} + \cdots + \sum_{i=1}^{c} a_{ni}w_{ni}\\
 &=\sum_{j=1}^{n}\sum_{i=1}^{c} a_{ji}w_{ji} \\
 &=\sum_{i=1}^{c}\sum_{j=1}^{n} a_{ji}w_{ji}\ .
\end {align*} 

Since $f$ and $g$ are mappings on $G$ that are homotopic to the constant map, the image of cycles must collapse. So for $m$ satisfying $1\leq m \leq c$ we know $N\vec w_m=\vec 0$ and $\vec 0 =M\vec w_m=(N+B)\vec w_m=N\vec w_m + B\vec w_m= \vec 0+B\vec w_m=\vec0$. Therefore $B\vec w_m=\vec0$. Thus we obtain 

\begin{align*}
\vec 0 &= B\vec w_m\\ 
&= w_{1m}\sum_{i=1}^{c} a_{1i}\vec w_i+ w_{2m}\sum_{i=1}^{c} a_{2i}\vec w_i+ \cdots + w_{nm}\sum_{i=1}^{c} a_{ni}\vec w_i\\
&=\sum_{j=1}^n\sum_{i=1}^{c} w_{jm}a_{ji}\vec w_i\\
&=\sum_{i=1}^{c}\sum_{j=1}^n w_{jm}a_{ji}\vec w_i\\
&=\Bigg(\sum_{j=1}^n w_{jm}a_{j1}\Bigg)\vec w_1+\Bigg(\sum_{j=1}^n w_{jm}a_{j2}\Bigg)\vec w_2+\cdots+\Bigg(\sum_{j=1}^n w_{jm}a_{jc}\Bigg)\vec w_c \ .
\end{align*}

Recall that the vectors in $W$ are linearly independent. Since we have a linear combination of linearly independent vectors that is equal to the zero vector, all coefficients must be equal to zero. More specifically, the $m^{th}$ coefficient is equal to zero. So for $m$ satisfying $1\leq m \leq c$  we know $\sum_{j=1}^n w_{jm}a_{jm}=0$. So it follows that $\sum_{i=1}^{c}\sum_{j=1}^{n} a_{ji}w_{ji}=0$. This is the trace of $B$. Therefore $\tr(B)=0$. 

It follows that $\tr(M)=\tr(N)+\tr(B)=\tr(N)+0=\tr(N)$, so the trace of the Oriented Markov Matrices of two maps HTC on a graph $G$ with a given permutation will be the same.
\end{proof}


\begin{thm}
	Given any graph $G$ and any permutation $\theta$ that does not fix any vertices, the Oriented Markov Matrix of any HTC map from $G$ to $G$ with permutation $\theta$ will have a trace of $-1$.
\end{thm}
\begin{proof}
By Lemma 5, for a given permutation, there exists a map such that the trace of its Oriented Markov Matrix is $-1$. By Lemma 6, any two maps on a graph that have the same permutation will have the same trace. Therefore, for any graph $G$ and any permutation $\theta$ that does not fix any vertices, the Oriented Markov Matrix of any map from $G$ to $G$ will have trace $-1$.
\end{proof}


\begin{thm}
\label{thm:OMMpowers}
If $M$ and $N$ are the Oriented Markov Matrices of two maps from the graph $G$ to itself which are HTC and have the same vertex permutation, then for any positive integer $r$
\[
	M^r=MN^{r-1}.
\]
\end{thm}

\begin{proof}
Let $M$ and $N$ be the Oriented Markov Matrices described above.
Under these two maps, the image of each edge can only differ by an integer number of times edges map around complete cycles.
Thus, we know that $M=N+B$, where the columns of $B$ are integer linear combinations of the vectors in $W$.
Notice, then, that
\[
	M^2=M(N+B)=MN+MB.
\]
It is clear that $MB=0$ because the columns of $B$ are linear combinations of vectors in $W$ and $M\vec{w}=\vec{0}$ if $\vec{w} \in W$.

So $M^2=MN$. Induction then gives $M^r=MN^{r-1}$.

\end{proof}

\begin{thm}
\label{thm:eqOMMs}
	
	Given an HTC map $f:G \to G$  with a permutation $\theta$ and Oriented Markov Matrix, $M$,  if $\theta^p$ is the identity permutation, then  $M^{p+1}=M$.
\end{thm}

\begin{proof}
	Let $M$ and $f$ be as described above.
	Recall, $M^{p+1}$ is the Oriented Markov Matrix associated with $f^{p+1}$.
	Let $T$ be the Oriented Markov Matrix associated with an HTC map $f_S$, with the same permutation as $f$ on $G$, whose image is a spanning tree.
	Recall from Theorem \ref{thm:OMMpowers} above that $M^{p+1}=M T^p$.
	
	Since the image of $G$ under $f_S$ is a spanning tree of $G$, the image of $G$ under $(f_S)^r$ will be the same spanning tree for any given $r$.
	In particular, the image of $G$ under $(f_S)^p$ is this spanning tree.
	Since there are no cycles in a tree, there is only one reduced path from $v_a$ to $v_b$.
	We know that $\theta^p$ fixes all vertices.
	Therefore, if $E_i$ is an edge in the spanning tree, $f_S^p(E_i)$ can be reduced to $E_i$. This means that the $i^{th}$ column of $T^p$ has $1$ in the $i^{th}$ entry and zeros for the other entries.
	
		Since the labeling of edges is arbitrary, let us label the edges in the spanning tree $E_1, E_2, \ldots E_{v-1}$ and the edges that are not a part of the 						spanning tree as $E_v, E_{v+1},\ldots, E_n$. 
	Let $E_a$ be an edge in the spanning tree. The $a^{th}$ column of $T^p$ will have an entry of $1$ in the $a^{th}$ component and entries of $0$ everywhere else.
	So the first $v-1$ columns of the matrix $M T^p= M^{p+1}$ are identical to the first $v-1$ columns of the matrix $M$.

	Given any edge $E_z$ that is not in the spanning tree, we know that there is a cycle in the graph $G$ that contains $E_z$ and such that every other edge belongs to the spanning tree. Let $\vec w_z$ be the vector that corresponds to this cycle. Since $f$ is HTC, $M\vec w_z=\vec 0$. This means that the column of $M$ corresponding to $E_z$ is a linear combination of the first $v-1$ columns. But note that the same argument shows that the column of $M^{p+1}$ corresponding to $E_z$ is exactly the same linear combination of the first $v-1$ columns of $M^{p+1}$. Thus the columns corresponding to $E_z$ in  $M$ and $M^{p+1}$ are equal. So 
$M^{p+1}=M$.

%
\end{proof}


\section{Periods of periodic orbits}

In this section we use the trace results from the previous section to prove our main results.

\begin{thm}
Suppose that $f:G \to G$ is HTC and permutes the $v$ vertices of $G$ with permutation $\theta$, where $\theta$ consists of one cycle. If $v$ is not a divisor of $2^k$, then $f$ has a periodic point of period $2^k$.
\end{thm}

\begin{proof}
Since $v$ is not a divisor of $2^k$, we know that $\theta^{2^k}$ does not fix any of the vertices.
So $M^{2^k}$ has a trace of $-1$.
Ergo there is at least one edge $E_{i_0}$ with a closed walk of length $2^k$ with negative orientation.
Since the orientation is negative, it cannot be a repetition of a shorter closed walk, as any shorter closed walk would have to be repeated an even number of times. Let $E_{i_0}E_{i_1}\cdots E_{i_{2^k}}$ denote the closed walk, with $E_{i_0}=E_{i_{2^k}}$.

We know that there is a closed subinterval $J \subseteq E_{i_0}$ that gets mapped onto $E_{i_0}$ by $f^{2^k}$. 
As pointed out before, the endpoints of $E_{i_0}$ might belong to other intervals, and we have to be careful that the closed walk is not describing one of the endpoints of $E_{i_0}$. 
However, since $J$ gets mapped onto $E_{i_0}$ with negative orientation, the point that is fixed by $f^{2^k}$ must be an interior point. Let $z$ denote this point. Since $z$ is fixed by $f^{2^k}$ and is an interior point of $E_{i_0}$, it must be the case that $f^j(z)$ is in the interior of $E_{i_j}$ for $0 \leq j \leq 2^k$. Since the walk is not the repetition of a shorter walk it must be the case that $z$ has minimum period of $2^k$.

\end{proof}

\vspace{2pc}

We say a closed walk from $E$ to itself is \emph{prime}, if it is not the concatenation of shorter walks from $E$ to itself. 

\begin{thm}
Suppose that $f:G \to G$ is HTC and permutes the $v$ vertices of $G$ with permutation $\theta$, where $\theta$ consists of one cycle. Suppose that $v=2^ks$, where  $s>1$ is odd, and $k\ge0$. Then for any $r>s$ there exists a periodic point of minimum period $2^k r$.

\end{thm}

\begin{proof}
Consider $f$ as described above, and let $\OMM(f) = M$. Since $f$ permutes the vertices by $\theta$, the map $f^{2^k}$ permutes the vertices by $\theta^{2^k}$, and $\OMM(f^{2^k}) = M^{2^k}$.
Thus, the vertices of $G$ all have minimum period $s$ under $f^{2^k}$, so none of them are fixed. 

By Theorem 3, $(M^{2^k})^{s+1}=M^{2^k s +2^k}=M^{2^k}$. We see by Theorem 1 that $\tr(M^{2^k})= \tr((M^{2^k})^{s+1})=-1$. 
So there is a closed walk of length $2^k$ from an edge to itself with negative orientation, and a closed walk of length $2^k (s+1)$ from that same edge to itself with negative orientation. Since $2^k$ is a power of 2, any repeated walk would have to be repeated an even number of times and therefore have positive orientation. Thus, the $2^k$-length walk is not a repetition of a smaller walk.
Since $s$ is odd, $s+1$ is even, so repeating the closed $2^k$-length walk $s+1$ times would have positive orientation.
Therefore, the walk of length $2^k (s+1)$ is not a repetition of the walk of length $2^k$.
So for all $r>s$ we can produce a closed walk of length $2^k r$ by repeating the $2^k$-length walk $r-s-1$ times and the walk of length $2^k(s+1)$ once. 

Although we have found a periodic point of period $2^k r$, this walk may be repetitive, and so we have not shown the existence of a point with minimum period $2^k r$.

If this closed walk of length $2^k r$  is repetitive, then we will  construct a non-repetitive closed walk of length $2^k r$ by rearranging the prime closed walks that comprise the repetitive walk. In what follows we will fix an $E$ that appears in the closed walk and consider prime closed walks to this particular edge.

Notice that if a closed walk is not prime it must be the concatenation of at least two possibly distinct prime closed walks. We will show that the walk of length $2^k r$ contains at least two prime walks from $E$ to itself. The one exception is if $r=s+1$ and we deal with that first.

First we consider the case when $r=s+1$. If the walk of length $2^k r$ is repetitive and does not contain two distinct prime walks, it must consist of a prime walk repeated an odd number of times. Let the length of this prime walk be $2^k t$. We can obtain a new walk of length $2^k r$ by first using this prime walk of length $2^k t$ and the repeating the walk of length $2^k$ $r-t$ times. This new walk is non-repetitive and has negative orientation, and so there must a periodic point with minimum period $2^k r$. We now consider the case when $r>s+1$.

If the closed walk of length $2^k$ is prime, then there must exist at least one other prime closed walk in the walk of length $2^k (s+1)$, since the latter is not a repetition of the former.
If the closed walk of length $2^k$ is not prime, then because it is not repetitive there are at least two distinct prime closed walks in the closed walk of length $2^k$.
In either case there are at least two distinct prime closed walks in the closed walk of length $2^k r$.

Since we are assuming the closed walk of length $2^k r$ is repetitive, each prime closed walk must exist in that walk at least twice. We will let $P_1$ denote a prime walk of shortest length that appears in the walk of length $2^k$.
Suppose the closed walk of length $2^k r$ above has prime closed walks $P_1,P_2,\ldots,P_i$. Suppose for each $j$ that  $P_j$ is in the closed walk of length $2^k r$ a total of $a_j$ times, where $a_j\geq 2$.
Since all prime closed walks must begin and end at $E$, we may arrange them in any order and still have a valid walk.
So we may arrange them so that $P_1$ is repeated $a_1$ times followed by $P_2$ repeated $a_2$ times, etc.
It is clear that this closed walk cannot be repetitive.
Ergo, we can create a non-repetitive closed walk of length $2^k r$ by rearranging the prime closed walks from the repetitive closed walk of length $2^k r$ created above. 

So there exists a non-repetitive closed walk of length $2^k r$.
We know that this walk implies the existence of a point with period $2^k r$, but because the vertices may appear in multiple intervals, it is still conceivable that this point is a vertex and consequently might have minimum period less than $2^k r$. Let $z \in E$ denote the periodic point with period $2^k r$. To complete the proof we must show that $z$ cannot be a vertex.

Since $r>s$ and $s \ge 3$, our construction starts by repeating the length-$2^k$ walk at least twice. This means that $P_1$ must appear at least twice. Let the length of $P_1$ be denoted by $l$. So our non-repetitive walk of length $2^k r$ has $E$ in the $1$, $l+1$ and $2l+1$ positions. Since $P_1$ is in the walk of length $2^k$ we know that $l \leq 2^k$. 

Let $v_a$ and $v_b$ denote the vertices that are endpoints of $E$. Suppose that $z=v_a$. Since the vertices cannot be mapped into the interior of $E$ it must be the case that $f^{l}(v_a)$ is either $v_a$ or $v_b$. We know that the period of $v_a$ is $2^k s$ and $2^k s > l$, so 
$f^{l}(v_a)$ must be $v_b$. Similarly, we know that $f^{2l}(v_a)$ is either $v_a$ or $v_b$. Since $v_b$ has period greater than $l$, it must be the case that $f^{2l}(v_a)$ is $v_a$. This implies that $v_a$ must have minimum period that is less than or equal to $2l$. But this gives a contradiction as we know that the minimum period of  $v_a$ is $2^k s$ and that $2l \leq 2^k 2< 2^k s$.

A similar argument shows that $z$ cannot be $v_b$. So $z$ must have minimum period $2^k r$.

\end{proof}

\section{Concluding remarks}

In this section we will compare our results to Sharkovsky's ordering and to the tree ordering in \cite{B3}.

The Sharkovsky ordering can be defined as follows: 
\begin{enumerate}
\item $2^l \triangleleft 2^k$ if $k \geq l$.
\item If $v=2^ks$, where $s>1$ is odd, then

\begin{enumerate}
\item $2^l \triangleleft v$, for all positive integers $l$.
\item $2^kr \triangleleft v$, where $r\geq s$ and $r$ is odd.
\item $2^lr \triangleleft v$, where $l>k$ and $r>1$ is odd.

\end{enumerate}
\end{enumerate}

To compare the various orderings we will state the tree ordering and the result of this paper using similar terms. First we re-state the results of this paper.

\begin{thm}
Suppose that $G$ is a graph with $v$ vertices and $f:G \to G$ is a map that is HTC and such that  the vertices of $G$ form one periodic orbit.
Then 
\begin{enumerate}
\item If $v=2^k$, then there must be periodic points of minimum period $2^l$ for any $l \leq k$.
\item If $v=2^ks$, where $s>1$ is odd, then

\begin{enumerate}
\item there are periodic points with minimum period $2^l$ for all positive integers $l$,
\item there are periodic points with minimum period $2^kr$ for any $r\geq s$ and $r$ is odd.
\item there are periodic points with minimum period $2^lr$ for all $l$ and $r$ satisfying: $l>k$, $r>1$ is odd, and $2^{l-k}r>s$.

\end{enumerate}
\end{enumerate}
\end{thm}

\begin{proof}
The statements involving points with period $2^l$ follow immediately from Theorem $4$. 

The last two statements follow immediately from Theorem $5$.

\end{proof}

We now give the corresponding result for trees, but first we need to introduce the concept of {\em removing $1$s from the right}.

This process of removing $1$s from the right can be described as follows.

\begin{enumerate}
\item Write $v$ in binary.
\item Change the rightmost $1$ in its expansion to zero. 
\item Repeat the process until you end with $0$.
\end{enumerate}

 For example, $31$ has binary expansion $11111$. Applying the process to this number yields the following binary expansions $11110$, $11100$, $11000$, $10000$ and $00000$, or in decimal notation $30$, $28$, $24$,  $16$ and $0$.

We can now state the theorem for trees.

\begin{thm} 

Let $T$ be a tree with $v$ vertices. Let $f: T \rightarrow T$ be a map with the property that the vertices form one periodic orbit. Then: 
\begin{enumerate}
\item If $v=2^k$, then there must be periodic points of minimum period $2^l$ for any $l \leq k$.
\item If $v=2^ks$, where $s>1$ is odd, then

\begin{enumerate}
\item there are periodic points with minimum period $2^l$ for all positive integers $l$,
\item there are periodic points with minimum period $2^kr$ for any $r\geq s$ and $r$ is odd.
\item there are periodic points with minimum period $2^lr$ for all $l$ and $r$ satisfying: $l>k$, $r>1$ is odd, and $2^{l-k}r>s$.

\item The map $f$  also has periodic orbits of any minimum period $m$ where $m$ can be obtained from $v$ by removing $1$s from the right of the binary expansion of $n$ and changing them to zeros.
\end{enumerate}
\end{enumerate}
\end{thm}

However, the above theorems shows that  if we look at the set of periods given by our results for graphs with $v$ vertices, set of periods given in \cite{B3} for trees with $v$ vertices and the set of integers that are forced by $v$ in the Sharkovsky ordering, they will differ by at most a finite number of integers, all of which will be less than $v$. All three orderings agree on the integers that are greater than $v$ and forced by $v$. For example if $v=30$, the periods that are less than $v$ and forced by Sharkovky's theorem are $1, 2, 4, 8, 16, 12, 20, 28, 24$; the periods that are less than $v$ and forced by the tree ordering are $1, 2, 4, 8, 16, 24, 28$; and the periods that are less than $v$ forced by Theorem $6$ are $1, 2, 4, 8, 16$. 
 
In \cite{S} and \cite{B2, B3} what are sometimes called the converses are shown. That is examples are constructed for each positive integer $v$ that have their set of minimum periods being exactly the set of periods given by the forcing relation. For HTC maps we do not have this. It is an interesting open question to ask whether there exists an HTC map of a graph with $v$ vertices such that the vertices form one periodic orbit and such that there does not exist a periodic point of period $m$ where $m$ is forced by $v$ in the tree ordering. For example, does there exist an HTC map of a graph with $30$ vertices that does not have a periodic points with minimum periods of $24$ or $28$.

\end{document}